\documentclass[12pt]{article}
\usepackage{amsmath}
\usepackage{amsfonts}
\usepackage{amssymb}
\usepackage{amsthm}
\usepackage{stackrel}

\usepackage{color}

\newtheorem{df}{Definition}[section]
\newtheorem{thm}{Theorem}[section]

\newcommand{\R}{{\rm I}\kern-0.18em{\rm R}}
\newcommand{\1}{{\rm 1}\kern-0.25em{\rm I}}
\newcommand{\E}{{\rm I}\kern-0.18em{\rm E}}
\newcommand{\p}{{\rm I}\kern-0.18em{\rm P}}
\makeatletter

\title{Outliers, the Law of Large Numbers, Index of Stability and Heavy Tails}
\author{Lev B. Klebanov\footnote{Department of Probability and Mathematical Statistics, MFF, Charles University, Prague, Czech Republic; e-mail: lev.klebanov@mff.cuni.cz}, Ashot V. Kakosyan\footnote{Yerevan State University, Yerevan, Armenia.} and Andrea Karlova\footnote{Institute of Information Theory and Automation, CAS, Prague, Czech Republic.}}
\date{}
\begin{document}
\maketitle

\begin{abstract}
We are trying to give a mathematically correct definition of outliers. Our approach is based on the distance between two last order statistics and appears to be connected to the law of large numbers.

\vspace{0.2cm}
\end{abstract}

\section{Introduction and the statement of the problem}\label{se1}
\setcounter{equation}{0}
Let us consider a notion of {\bf outliers}. Wikipedia, the free encyclopedia defines outliers in the following way: ``In statistics, an outlier is an observation point that is distant from other observations. An outlier may be due to variability in the measurement or it may indicate experimental error; the latter are sometimes excluded from the data." Some points from this definition need essential clarification. Namely, the main problem is to define what sense have the words ``distant from other observations". One may understand them in the sense that the absolute value of the difference between this observation and empirical mean is larger that a number (say $k$) of empirical standard deviations $s$. For example, we say $X_j$ is an outlier if $|X_j-\bar{x}|>k s$, where $\bar{x} = \sum_{i=1}^{n}X_i/n$ and $s= (\sum_{i=1}^{n} (X_i-\bar{x})^2/n)^{1/2}$. Here $X_1, \ldots ,X_n$ are independent identically distributed (i.i.d.) observations. This notion will not lead to distributions with heavy tails and is, in a sense, miss leading. It was criticized in \cite{Kl, KV}. Here we consider another approach. Namely, outlier in our sense is an extremal observation which is larger in its absolute value than $1/\kappa$ times previous extremal observation.

Let us give precise definition.
\begin{df}\label{de1}
Let $X_1, \ldots ,X_n$ be i.i.d. random variables, and $X_{(1)}, \ldots , X_{(n)}$ be the observations ordered in its absolute values (from minimal to maximal). We say $X_{(n)}$ is an outlier of order $1/\kappa$ if $X_{(n-1)}\leq \kappa X_{(n)}$, where $\kappa \in (0,1)$ is a fixed number.
\end{df} 
In this paper we find a boundary for probability of outlier of order $1/\kappa$ and show its connection with the index of stability.

\section{Main results}\label{se2}
\setcounter{equation}{0}

Let us calculate the probability of $X_{(n)}$ to be an outlier of order $1/\kappa$ for a fixed $\kappa \in (0,1)$. Suppose that $X_1, \ldots ,X_n$ are i.i.d. random variables, and $X_{(1)}, \ldots , X_{(n)}$ are the observations ordered in its absolute values. Suppose that random variable $|X_1|$ has absolute continuous distribution function $F(x)$, and $p(x)$ is its density. Denote by $p_{n-1,n}(x,y)$ the common density of $X_{(n-1)}$ and $X_{(n)}$. We have (see, for example, \cite{Dav}) 
\begin{equation}\label{eq1}
p_{n-1,n}(x,y) = n(n-1) F^{n-2}(x)p(x) p(y),
\end{equation}
for $x \leq y$. Therefore the probability of the event that $X_{(n-1)} \leq \kappa X_{(n)}$ is
\begin{equation}\label{eq2}
\p\{X_{(n-1)} \leq \kappa X_{(n)}\} = n \int_{0}^{\infty}F^{n-1}(\kappa y)p(y)dy.
\end{equation}
Let us try to study limit behavior of the probability (\ref{eq2}) for large values of sample size $n$. We have
\begin{equation}\label{eq3}
\p\{X_{(n-1)} \leq \kappa X_{(n)}\}=n \int_{0}^{\infty}F^{n-1}(x)\kappa p(\kappa x)\frac{p(x)}{\kappa p(\kappa x)}dx. 
\end{equation}
Assume that 
\begin{equation}\label{eq4}
\lim_{x \to 0}F^{n}(\kappa x)\frac{p(x)}{p(\kappa x)} =0. 
\end{equation}
Integrating by parts in (\ref{eq3}) gives us
\begin{equation}\label{eq5}
\begin{aligned}
\p\{X_{(n-1)} \leq \kappa X_{(n)}\}= \lim_{x \to \infty}\frac{p(x)}{\kappa p(\kappa x)}-\\ -\int_{0}^{\infty}F^{n}(\kappa x)\Bigl(\frac{p^{\prime}(x)}{\kappa p(\kappa x)}- \frac{p(x) p^{\prime}(\kappa x)}{p^2(\kappa x)} \Bigr) dx
\end{aligned}
\end{equation}
If the function
\[ \Bigl(\frac{p^{\prime}(x)}{\kappa p(\kappa x)}- \frac{p(x) p^{\prime}(\kappa x)}{p^2(\kappa x)} \Bigr) \]
is integrable over $(0,\infty)$ then 
\[ \int_{0}^{\infty}F^{n}(\kappa x)\Bigl(\frac{p^{\prime}(x)}{\kappa p(\kappa x)}- \frac{p(x) p^{\prime}(\kappa x)}{p^2(\kappa x)} \Bigr) dx \to 0 \]
as $n \to \infty$.
Therefore,
\begin{equation}\label{eq6}
\lim_{n \to \infty}\p\{X_{(n-1)} \leq \kappa X_{(n)}\}= \lim_{x \to \infty}\frac{p(x)}{\kappa p(\kappa x)},
\end{equation}
assuming that the limit in right-hand side of (\ref{eq6}) exists.

Finally, we obtain the following result.
\begin{thm}\label{th1}
 Suppose that $X_1, \ldots ,X_n$ are i.i.d. random variables, and $X_{(1)}, \ldots , X_{(n)}$ are the observations ordered in its absolute values. Let random variable $|X_1|$ has absolute continuous distribution function $F(x)$, and let $p(x)$ be its density. Suppose that $p(x)$ is regularly varying function of index $-(\alpha+1)$ on infinity, function 
 \[ \Bigl(\frac{p^{\prime}(x)}{\kappa p(\kappa x)}- \frac{p(x) p^{\prime}(\kappa x)}{p^2(\kappa x)} \Bigr) \]
 is integrable over $(0,\infty)$, and 
 \[ \lim_{x \to 0}F^{n}(\kappa x)\frac{p(x)}{p(\kappa x)} =0.  \]
Then
\begin{equation}\label{eq7}
\lim_{n \to \infty}\p\{X_{(n-1)} \leq \kappa X_{(n)}\}=\kappa^{\alpha}.
\end{equation}
\end{thm}
\begin{proof}
The statement of the Theorem follows from considerations given above and from the definition of regularly varying function (see, for example \cite{Sen}).
\end{proof}

\section{Connection to the law of large numbers and statistical definition of stability index}
\setcounter{equation}{0}

Theorem \ref{th1} shows that there is a connection between stable distribution and the probability of presence of $1/\kappa$ outliers.
Namely, the condition ``$p(x)$ is regularly varying function of index $-(\alpha+1)$ on infinity" implies that corresponding random variables $X_1, \ldots ,X_n$ belong to the region of attraction of $\alpha$-stable distribution. The probability (\ref{eq7}) is defined by index $\alpha$ in unique way, and increase with decreasing $\alpha$. 

For the first glance, it is not clear why there is no law of large numbers in the case of $\alpha \in (0,1)$. Really, in the case of symmetric distributions, it seems to be possible, that large positive observations may be compensated by corresponding negative observations, coming into empirical mean with the same probability as positive.
For $\alpha \in (0,1)$ the limit probability for $X_{(n-1)}$ to be less that $\kappa X_{(n)}$ is greater than $\kappa$ itself. It shows, that very often the ``maximal" observation $X_{(n)}$ cannot be ``compensated" by smaller observations. It gives us an intuitive explanation of why there is no law of large numbers for the case of $\alpha \in (0,1)$. 

Is it possible to use the relation (\ref{eq7}) to define the stability index $\alpha$? Of course, it is possible theoretically, but is impossible statistically, because we cannot pass to limit for any large (but finite) number $n$ of observations. However, the probability $\p\{X_{(n-1)}< \kappa X_{(n)}\}$ (for fixed $\kappa$ and fixed $n$) may be statistically estimated. Such probability does not define ``true" value of $\alpha$, however, small value of such estimator for $\alpha$ shows that empirical mean is not close to any constant at least for corresponding values of $n$.

\section{Conclusion}
Definition \ref{de1} give us ``working" notion of outliers. This notion appears to be connected with index of stability. It may be used for statistical estimation of suitable variant of this index. Theorem \ref{th1} provides empirical explanation of the absence of the law of large numbers for the case of stability index smaller than 1.

\section*{Acknowledgment}
The work was partially supported by Grant GACR 16-03708S.

\end{document}